\documentclass{acm_proc_article-sp}

\usepackage{url}
\usepackage{amsfonts}

\newtheorem{lemma}{Lemma}
\newtheorem{definition}{Definition}
\newtheorem{theorem}[lemma]{Theorem}
\newtheorem{corollary}[lemma]{Corollary}

\usepackage{todonotes}

\def\FPRAS{\mathrm{FPRAS}}
\def\FPAUS{\mathrm{FPAUS}}

\def\P{\mathrm{P}}
\def\NP{\mathrm{NP}}

\def\FP{\mathrm{FP}}
\def\sharpP{\mathrm{\#P}}
\def\sharpPcomp{\mathrm{\#P-complete}}

\def\FPAUS{\mathrm{FPAUS}}

\def\sharp3SAT{\mathrm{\#3SAT}}

\begin{document}
\title{Packing tree degree sequences
%\titlenote{(Does NOT produce the permission block, copyright information nor page 
%numbering). For use with ACM\_PROC\_ARTICLE-SP.CLS. Supported by ACM.}
}
%\subtitle{[Extended Abstract]
%\titlenote{A full version of this paper is available as
%\textit{Author's Guide to Preparing ACM SIG Proceedings Using
%\LaTeX$2_\epsilon$\ and BibTeX} at
%\texttt{www.acm.org/eaddress.htm}}

\numberofauthors{4} %  in this sample file, there are a *total*
\author{
% 1st. author
\alignauthor
Krist\'of B\'erczi\thanks{Research is supported by a grant (no.\ K
    109240) from the National Development Agency
    of Hungary, based on a source from the Research and Technology Innovation
    Fund.
}\\
       \affaddr{Department of Operations Research, 
E\"otv\"os Lor\'and University}\\
       \affaddr{P\'azm\'any P\'eter s\'et\'any 1/c}\\
       \affaddr{1117 Budapest, Hungary}\\
       \email{berkri@cs.elte.hu}
% 2nd. author
\alignauthor
Zolt\'an Kir\'aly\thanks{Research is supported by a grant (no.\ K
    109240) from the National Development Agency
    of Hungary, based on a source from the Research and Technology Innovation
    Fund.
}\\
\affaddr{Department of Computer Science and
  Egerv\'ary Research Group (MTA-ELTE),
E\"otv\"os Lor\'and University}\\
       \affaddr{P\'azm\'any P\'eter s\'et\'any 1/c}\\
       \affaddr{1117 Budapest, Hungary}\\
       \email{kiraly@cs.elte.hu}
\and
% 3nd. author
\alignauthor
Changshuo Liu\\
       \affaddr{Budapest Semesters in Mathematics}\\
       \affaddr{Bethlen G\'abor t\'er 2}\\
       \affaddr{1071 Budapest, Hungary}\\
       \email{cl20@princeton.edu}
% 4th. author
\alignauthor
Istv\'an Mikl\'os\thanks{Secondary affiliation: SZTAKI, 1111 Budapest, L\'agym\'anyosi u. 11, Hungary}\\
       \affaddr{R\'enyi Institute}\\
       \affaddr{Re\'altanoda u. 13-15}\\
       \affaddr{1053 Budapest, Hungary}\\
       \email{miklos.istvan@renyi.mta.hu}
% There's nothing stopping you putting the seventh, eighth, etc.
% author on the opening page (as the 'third row') but we ask,
% for aesthetic reasons that you place these 'additional authors'
% in the \additional authors block, viz.
% Just remember to make sure that the TOTAL number of authors
% is the number that will appear on the first page PLUS the
% number that will appear in the \additionalauthors section.
}
\maketitle

\begin{abstract}
A degree sequence $D=d_1,d_2,\dots, d_n$ is a series on non-negative integers. A degree sequence is graphical if there exists a vertex labeled graph $G$ in which the degree of vertex $v_i$ is exactly $d_i$ for $i=1,\dots,n$. The graph $G$ is called a realization of $D$. The color degree matrix problem, also known as edge disjoint realization, edge packing or graph factorization problem, is the following: given a $c \times n$ degree matrix $D=\{\{d_{1,1},d_{1,2},\dots, d_{1,n} \},\{d_{2,1},d_{2,2},\dots,d_{2,n} \},\dots \{d_{c,1},d_{c,2},\allowbreak \dots, \break d_{c,n}\}\}$, in which each row of the matrix is a graphical degree sequence, decide if there exists pairwise edge-disjoint realizations of the degree sequences. Such set of edge disjoint graphs is called a realization of the degree matrix. A realization can also be presented as an edge colored simple graph, in which the edges with a given color form a realization of the degree sequence in a given row of the color degree matrix.

It is known that the color degree matrix problem is $\NP$-complete even if the number of colors is three and the degrees on each vertex sum up to $n-1$, that is, when a decomposition of the complete graph is required into subgraphs with prescribed degrees; and it is also $\NP$-complete when the number of colors is two and the sum of the degrees on some of the vertices is less than $n-1$. However, special cases that are computationally tractable are also of interest. A classical result of Kundu \cite{kundutree} shows that deciding if two tree degree sequences have edge disjoint realizations is in P.

Motivated by the aforementioned result, we consider special cases of the two tree degree sequences problem. We show that if two tree degree sequences do not have common leaves then they always have edge-disjoint caterpillar realizations. By using a probabilistic method, we prove that two tree degree sequences always have edge-disjoint realizations if each vertex is a leaf in at least one of the trees. This theorem can be extended to more trees: we show that the edge packing problem is in P for an arbitrary number of tree sequences with the property that each vertex is a non-leaf in at most one of the trees. 

We also consider the following variant of the degree matrix problem: given two degree sequences $D_1$ and $D_2$ such that $D_2$ is a tree degree sequence, decide if there exists edge-disjoint realizations of $D_1$ and $D_2$ where the realization of $D_2$ is not necessarily a tree. We show that this problem is already $\NP$-complete.

Counting, or just estimating the number of distinct realizations of degree sequences is challenging in general. We show that efficient approximations for the number of solutions as well as an almost uniform sampler exist for two tree degree sequences if each vertex is a leaf in at least one of the trees.
\end{abstract}

\section{Introduction}

Packing degree sequences is related to discrete tomography. The central problem of tomography is to reconstruct spatial objects from lower dimensional projections. The discrete 2D version is to reconstruct a colored grid from vertical and horizontal projections. In the simplest version, this problem is to reconstruct the coloring of an $n \times m$ grid with the requirement that each row and column has a specific number of entries for each color. Such colored matrix can be considered as a factorization of the complete bipartite graph $K_{n,m}$. Indeed, for each color $c_i$, the 0-1 matrix obtained by replacing $c_i$ to 1 and all other colors to 0 is an adjacency matrix of a simple bipartite graph such that the disjoint union of these simple graphs is $K_{n,m}$. The prescribed number of entries for each color are the degrees of the simple bipartite graphs. Therefore, an equivalent problem is to give a factorization of the complete bipartite graph into subgraphs with prescribed degree sequences.

It is also possible to consider the non-bipartite version of the graph factorization problem. Obviously, the sum of the degrees for each vertex must be $n-1$ when the complete graph $K_n$ is factorized. Therefore, if there are $k$ degree sequences, the last degree sequence is uniquely determined by the first $k-1$ degree sequences. When $k=2$, the problem is reduced to the degree sequence problem, and can be solved in polynomial time \cite{H65,H55}. When $k=3$, the problem already becomes $\NP$-complete \cite{dgm2009}. However, special cases are polynomially solvable. Such a special case is when one of the degree sequences is almost regular, that is, any two degrees differ at most by 1 \cite{kundu}.

In this paper we consider the case when $k=3$ and two of the degree sequences are tree degree sequences. It was already known that this case is tractable \cite{kundutree}. Here we present a new result considering special, caterpillar realizations. Another alternative proof is given for a special subclass of pairs of tree degree sequences that can be extended to an arbitrary number of sequences. The size of the solution space and sampling from it is also discussed. As a negative result, we show that deciding the existence of edge-disjoint realizations for two degree sequences $D_1$ and $D_2$ is $\NP$-complete even if $D_2$ is a tree degree sequence (but its realization do not have to be a tree). 

\section{Preliminaries}\label{sec:pre}

In this section we give the definitions and lemmas needed to state the theorems. The central problem in this paper is the color degree sequence problem.

\begin{definition}
A \textbf{degree sequence} $D=d_1,d_2,\dots, d_n$ is a series of non-negative integers. A degree sequence is \textbf{graphical} if there is a vertex labeled simple graph $G$ in which the degrees of the vertices are exactly $D$. Such graph $G$ is called a \textbf{realization} of $D$. The color degree matrix problem is the following: given a $c \times n$ degree matrix $D=\{\{d_{1,1},d_{1,2},\dots, d_{1,n} \},\{d_{2,1},d_{2,2},\dots,d_{2,n} \},\dots, \{d_{c,1},d_{c,2},\dots,\break d_{c,n}\}\}$, in which each row of the matrix is a degree sequence, decide if there is an ensemble of edge disjoint realizations of the degree sequences. Such a set of edge disjoint graphs is called a \textbf{realization} of the degree matrix. Given two degree sequences $D=d_1,d_2,\dots,d_n$ and $F=f_1,f_2,\dots,f_n$, their \textbf{sum} is defined as $D+F= d_1+f_1,d_2+f_2, \dots, d_n +f_n$.
\end{definition}

For sake of completeness, we define tree degree sequences, path sequences and caterpillars.

\begin{definition}
Let $D=d_1,d_2,\dots, d_n$ be a degree sequence. Then $D$ is called a \textbf{tree sequence} if $\sum_{i=1}^n d_i = 2n-2$ and each degree is positive. If all of the degrees are $2$ except two of them which are $1$, then $D$ is called a \textbf{path sequence}. A tree is a \textbf{caterpillar} if its non-leaf vertices span a path.
\end{definition}

We will use the following complexity classes later on.

\begin{definition}
A decision problem is in $\NP$ if a non-deterministic Turing Machine
can solve it in polynomial time. An equivalent definition is that a
witness proving the ``yes'' answer to the question can be verified in
polynomial time. A counting problem is in $\sharpP$ if it asks for the
number of witnesses of a problem in $\NP$. A counting problem in $\sharpP$ is in $\FP$ if there is a polynomial running time algorithm
which gives the solution. It is $\sharpPcomp$ if any problem in $\sharpP$ 
can be reduced to it by a polynomial-time counting reduction.
\end{definition}
% Two other classes,
% $\FPRAS$ and $\FPAUS$, concern the approximability of the number of witnesses.

\begin{definition}
A counting problem in $\sharpP$ is in $\FPRAS$ (\textbf{Fully Polynomial
Randomized Approximation Scheme}) if there exists a randomized
algorithm such that for any problem instance $x$, and $\epsilon, \delta > 0$,
it generates an approximation $\hat{f}$ for the solution $f$, satisfying
\begin{equation*}
P\left(\frac{f}{1+\epsilon} \leq \hat{f} \leq f(1+\epsilon)\right) \geq 1 - \delta, 
\end{equation*}
and the algorithm has a time complexity bounded by a polynomial of
$|x|$, $1/\epsilon$ and $-\log(\delta)$.
\end{definition}

The total variational distance $d_{TV}(p,\pi)$ between two discrete
distributions $p$ and $\pi$ over the set $X$ is defined as
\begin{equation*}
d_{TV}(p,\pi) := \frac{1}{2}\sum_{x \in X} |p(x) - \pi(x)|
\end{equation*}

\begin{definition}
A counting problem in $\sharpP$ is in $\FPAUS$ (\textbf{Fully Polynomial Almost Uniform Sampler}) if there exists a randomized
algorithm such that for any instance $x$, and $\epsilon > 0$, it
generates a random element of the solution space following a distribution $p$ satisfying
\begin{equation*}
d_{TV}(p,U) \leq \epsilon, 
\end{equation*}
where $U$ is the uniform distribution over the solution space,
and the algorithm has a time complexity bounded by a polynomial of $|x|$ and
$-\log(\epsilon)$. 
\end{definition}

The following technical lemma will be used later for constructing edge-disjoint caterpillar realizations.

\begin{lemma}\label{hamil}
For $n \ge 4$, there exists two edge-disjoint Hamiltonian paths in the complete graph $K_n$ whose ends are pairwise different.
\end{lemma}
\begin{proof}
  Let $V=\{1,2,\ldots,n\}$, and let the first Hamiltonian path be $1,2,3\ldots,n$. We are going to show by induction that there is a second  Hamiltonian path $H$ starting at $2$, ending at $3$ and using no edge between consecutive integers.
  For $n=4$ the path $H=2,4,1,3$ does the job. Suppose $n>4$ and we have a path $H'$ on vertices $1,\ldots,n-1$ between 2 and 3. Since it has at least three edges, there is an edge $ij$ where $i,j<n-1$. Replace this edge by two edges $in$ and $nj$ for getting the desired path $H$. 
\end{proof}

\section{Packing trees}

First we consider the problem of packing two tree degree sequences without common leaves. 

\begin{theorem}
\label{theo:caterpillar}
\begin{sloppypar}
Let $D = d_1, d_2, \dots, d_n$ and ${F = f_1, f_2, \ldots, f_n}$ be two tree degree sequences such that ${\min_{i}\{d_i + f_i\} \ge 3}$. Then $D$ and $F$ have edge disjoint caterpillar realizations.
\end{sloppypar}
\end{theorem}
\begin{proof}
The proof is by induction on $n$. Observe that the smallest possible $n$ is $4$ to accommodate at least $4 = 2 \times 2$ leaves (note that each tree has at least two leaves). For $n=4$, the only possible pair of degree sequences is $(2, 2, 1, 1)$ and $(1, 1, 2, 2)$. By Lemma~\ref{hamil}, these sequences have edge disjoint realizations.

If $n > 4$ and both $D$ and $F$ are path sequences, then there exists edge disjoint Hamiltonian paths, according to Lemma~\ref{hamil}. 

So we may suppose that not both are path sequences. As the sum of the degrees in $D+F$ is $4n-4$, there are at least four indices where $d_j+f_j=3$, it is easy to check that we can select indices $i$ and $j$ such that, possibly after reversing $D$ and $F$, we have $d_i\ge 3$, $d_j=1$ and $f_j=2$.

Modify $D$ and $F$ by removing $d_j$ and $f_j$ and decreasing $d_i$ by $1$. This modified $D'$ and $F'$ are tree degree sequences without common leaves on $n-1$ vertices, therefore, by induction, $D'$ and $F'$ have edge disjoint caterpillar realizations,  $T'_1$ and $T'_2$. Modify $T'_1$ and $T'_2$ as follows. Add back vertex $v_j$ and connect it to vertex $v_i$ in $T'_1$. The so obtained $T_1$ is a realization of $D$. Take a path $P$ in $T'_2$ containing all non-leaf vertices and two leaves. Observe that $P$ has at least 3 edges, since otherwise $F$ has $n-2$ leaves, so $D$ has only two, contradicting to $d_i\ge 3$. Hence $P$ has an edge $v_kv_\ell$ such that $k\ne i$ and $\ell\ne i$. For constructing $T_2$, replace edge  $v_kv_\ell$ of $T'_2$ by two edges, $v_kv_j$ and $v_jv_\ell$. The tree $T_2$ thus obtained is a caterpillar, edge disjoint from $T_1$ and is a realization of $F$.
\end{proof}

The theorem implicitly states that if two degree sequences do not share common leaves then their sum is graphical. If the two trees have common leaves, their sum is not necessarily graphical. The simplest example for it is the degree sequences
$$
\begin{array}{cccc}
D=& 2, & 1, & 1\\
F=& 2, & 1, & 1
\end{array}
$$
Observe that the largest degree in $D+F$ is $4$, and there are only $3$ vertices.
%\todo[inline]{An example would be nice.}

However, if their sum happens to be graphical then they also have edge disjoint realizations, as was shown by Kundu in \cite{kundutree}.

\begin{theorem} 
\begin{sloppypar}\cite{kundutree} Let $D = d_1, d_2, \dots, d_n$ and ${F = f_1, f_2, \dots, f_n}$ be two tree degree sequences. Then there exist edge disjoint tree realizations of $D$ and $F$ if and only if $D+F$ is graphical.\label{theo:kundutree}
\end{sloppypar}
\end{theorem}

However, there are tree degree sequences that have edge disjoint tree realizations but do not have edge disjoint caterpillar realizations. For example, consider the following tree degree sequences
$$
\begin{array}{cccccccccccc}
D=& 5, & 2, & 2, & 2, & 2, & 2, & 1, & 1, & 1, & 1, & 1 \\
F=& 5, & 2, & 2, & 2, & 2, & 2, & 1, & 1, & 1, & 1, & 1 \\
\end{array}
$$
They have edge disjoint realizations, according to Theorem~\ref{theo:kundutree} (see also Fig.~\ref{fig:twotree}), since their sum is graphical. 
 We claim that they do not have edge disjoint caterpillar realizations. To see this, observe that in any caterpillar realization, the degree $5$ vertices must be connected to at least $3$ leaves. However, there are only $5$ vertices that are leaves in any of the trees, showing that any pair of caterpillar realizations will share at least one edge.
\begin{figure}
\includegraphics[scale=0.33]{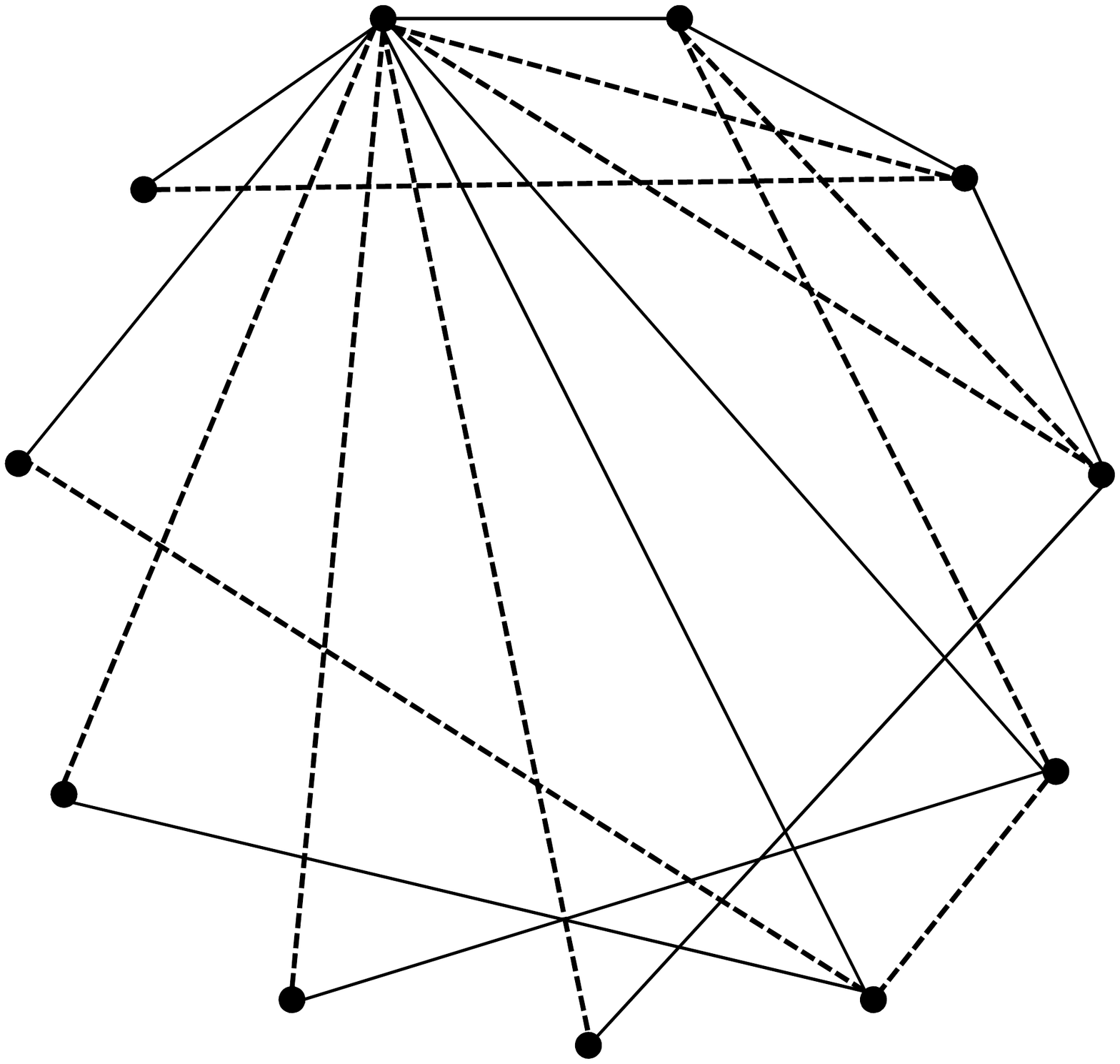}
\caption{Edge disjoint realization of two degree sequences, both of them are $5,2,2,2,2,2,1,1,1,1,1$}.\label{fig:twotree}
\end{figure}
 
Theorem~\ref{theo:caterpillar} considered the case when the leaf vertices of the degree sequences do not coincide. Now we turn to the opposite end, namely when each vertex is a leaf in at least one of the sequences.

\begin{theorem}\label{theo:random}
Let $D = d_1, d_2, \dots, d_n$ and $F = f_1, f_2, \dots, f_n$ be tree degree sequences such that $\min(d_i,f_i)=1$ for all $i$. Let $T_1$ and $T_2$ be random realizations of $D$ and $F$ uniformly distributed. Then the expected number of common edges of $T_1$ and $T_2$ is $1$.
\end{theorem}

\begin{proof}
The proof is based on the following lemma.

\begin{lemma}
Let $T$ be a random realization of the tree degree sequence $D= d_1, d_2, \dots, d_n$. Then the probability that there is an edge between $v_i$ and $v_j$ is
\begin{equation*}
\frac{d_i+d_j-2}{n-2}.
\end{equation*}
\end{lemma}
\begin{proof}
It is well known that the number of trees with a given degree sequence is
\begin{equation}
\frac{(n-2)!}{\prod_{k=1}^n (d_k-1)!}.\label{eq:number-of-trees}
\end{equation}
Let $T'$ denote those trees in which $v_i$ and $v_j$ are connected. Let $f$ be a mapping from $T'$ to the trees with degree sequence
$$
d_1,\dots,d_{i-1}, d_{i+1},\dots, d_{j-1},d_{j+1},\dots, d_n, d_i+d_j-2
$$
obtained by joining $v_i$ and $v_j$ to a common vertex. The function $f$ is surjective and each tree is an image ${d_i + d_j - 2 \choose d_i-1}$ times. Therefore the number of trees in which $v_i$ is connected to $v_j$ is
\begin{eqnarray}
&\frac{(n-3)!}{(d_i+d_j-3)!\prod_{k\ne i,j} (d_k-1)!} \frac{(d_i +d_j -2)!}{(d_i-1)!(d_j-1)!} =&\nonumber\\
&\frac{(d_i+d_j-2)(n-3)!}{\prod_{k=1}^n(d_k-1)!}.& \label{eq:v_i-v_j-trees}
\end{eqnarray} 
The probability that $v_i$ and $v_j$ is connected is the ratio of \eqref{eq:v_i-v_j-trees} and \eqref{eq:number-of-trees}, which is indeed
$$
\frac{d_i+d_j-2}{n-2},
$$ 
thus concluding the proof of the lemma.
\end{proof}

Now we turn to the proof of the theorem. Let $D$ and $F$ be the two degree sequences satisfying that each vertex is a leaf in at least one of the trees. Define
\begin{eqnarray*}
\mathcal{A} & := & \{i \;|\; d_i > 1 \wedge f_i = 1 \}, and \\
\mathcal{B} & := & \{i \;|\;  d_i = 1 \wedge f_i > 1 \}.  
\end{eqnarray*}
Note that there might be parallel edges in the two trees only between these two sets. The expected number of parallel edges is then
\begin{eqnarray*}
&\sum_{i \in \mathcal{A}} \sum_{j\in \mathcal{B}} \frac{(d_i-1)(f_j-1)}{(n-2)^2}  = 
\sum_{i \in \mathcal{A}} \frac{d_i-1}{n-2} \sum_{j\in \mathcal{B}} \frac{f_j-1}{n-2}=&
\nonumber \\
&\sum_{i = 1}^n \frac{d_i - 1}{n-2} \sum_{j = 1}^n \frac{f_j-1}{n-2}  =  1,&
\end{eqnarray*}
since $d_i=1$ for all $i \in \bar{\mathcal{A}}$, $f_j=1$ for all $j \in \bar{\mathcal{B}}$, and the sum of the degrees decreased by $1$ is $n-2$ for any tree degree sequence. This finishes the proof of the theorem.
\qed
\end{proof}

Theorem~\ref{theo:random} implies a characterization of realizability for a subclass of tree degree sequences.

\begin{corollary} \label{cor:not-star-tree}
Let $D = d_1,\dots,d_n$ and $F = f_1,\dots,f_n$ be tree degree sequences such that each vertex is a leaf in at least one of them. Then $D$ and $F$ have edge-disjoint tree realizations if and only if $d_i < n-1$ and $f_i < n-1$ for all $i$.
\end{corollary}
\begin{proof}
If $\max_i\{d_i\} = n-1$ or $\max_i\{f_i\} = n-1$ then $D+F$ is not graphical.
On the other hand, if none of the trees is a star, then there are four distinct indices such that $i_1,i_2 \in \mathcal{A}$ and $j_1,j_2\in\mathcal{B}$. Then there exists a pair of trees $T_1$ and $T_2$ such that both trees contain edges $(v_{i_1},v_{j_1})$ and $(v_{i_2},v_{j_2})$ and $T_1$ realizes $D$ while $T_2$ realizes $F$. Indeed, the degree $1$ vertices can be connected to any of the non-leaf vertices. This means that there are trees having at least $2$ common edges, which is above the average. Hence there must be a pair of trees with less than average number of common edges. That is, they are edge disjoint realizations. 
\end{proof}

This theorem will be useful also at generating random realizations, see the next section.

Similar theorem holds for arbitrary number of tree sequences.
We need a preliminary lemma (with $V=\{v_1,\ldots,v_n\}$).

\begin{lemma}\label{hart-of-sloppypar}
  Let $D=d_1,\dots,d_n$ be a tree degree sequence, $n>m>2$ and $U=\{v_i \;|\; d_i>1\}$. Suppose  $V_1,\ldots,V_{m-1}$ are pairwise disjoint sets in $L=V\setminus U$. Suppose further that $|U|>1, |V_1|>1,\ldots,|V_{m-1}|>1$ and $d_i\le n-m$ for all $i$. Then there is a tree $T$ realizing $D$, such that for all $j$ its restriction to $U\cup V_j$ is a non-star tree.
\end{lemma}
\begin{proof}
  For any tree realization $T$, its restriction to $U\cup V_j$ is a tree because outside  $U$  there are only leaves. In the case $|U|>2$ we claim that there is a tree realization $T$ such that its restriction to $U$ is not a star. Indeed, if $T'$ restricted to $U$ is a star centered at $u\in U$, then by the degree bound there is a leaf $w\in L$ not connected $u$, call its neighbor $u_1\in U$. Let $u_2$ be a third vertex of $U$. Replacing edges $uu_2$ and $u_1w$ by edges  $u_1u_2$ and $uw$ gives another tree realization $T$, whose restriction to $U$ is not a star.

  For the case $|U|=2$ let $U=\{v_i,v_j\}$ and connect first $v_1$ to $v_j$. Now $d_i+d_j=n$, so $d_i\ge m$ and $d_j\ge m$. For each $k\le m-1$ connect one vertex of $V_k$ to $v_i$ and another one to $v_j$. The remaining leaves in  $L$ can be distributed easily, connect any $d_i-m$ of them  to $v_i$ and the remainder to $v_j$ giving the aimed tree realization. 
\end{proof}

\begin{theorem}\begin{sloppypar}
Let $D_1,D_2,\dots, D_m$ be tree degree sequences with $D_i = d_{i,1},d_{i,2},\dots, d_{i,n}$ such that each vertex is a leaf in all except at most one of them. Then $D_1,D_2,\dots, D_m$ have edge disjoint realizations if and only if $\max_{i,j}\{d_{i,j}\} \le n-m$. 
\end{sloppypar}
\end{theorem}
\begin{proof}
  Necessity is clear as $D_1+D_2+\dots+D_m$ is not graphical if $\max_{i,j}\{d_{i,j}\} > n-m$.

  The statement is trivial when $m=1$, if $m=2$ then it is equivalent to  Corollary~\ref{cor:not-star-tree}, so we may suppose $m>2$.

We give a constructive proof for the other direction. First a trial solution is built which might contain parallel edges, then these parallel edges are eliminated to get an edge disjoint realization.

Let $V_i$ denote the subset of vertices on which the degrees in $D_i$ are larger than $1$. Note that $\{V_1,V_2,\dots,V_m\}$ forms a subpartition of $V$ and $|V_i|\geq 2$ for each $i=1,\dots,m$. For a degree sequence $D_i$, construct a trial tree $\tilde{T}_i$ by using Lemma \ref{hart-of-sloppypar}, which ensures that the subtree on vertices $V_i \cup V_k$ is a non-star tree for any $k\ne i$.

 From the trial solution, which might contain several parallel edges, a final solution is built in the following way. While there exists a pair of indexes $(i,k)$ such that there is one or more parallel edges between $V_i$ and $V_k$, do the following.
Let $\tilde{T}_{i,k}$ denote the subtree of the tree $\tilde{T}_i$ on vertices $V_i \cup V_k$ and let
$\tilde{D}_{i,k}$ denote its degree sequence. By Corollary~\ref{cor:not-star-tree}, $\tilde{D}_{i,k}$ and $\tilde{D}_{k,i}$ have edge disjoint tree realizations. Replace $\tilde{T}_{i,k}$ and $\tilde{T}_{k,i}$ by such realizations. This removes all parallel edges between $V_i$ and $V_k$ because $\tilde{T}_j$ has no edge between these sets if $j\ne i,\; j\ne k$.
\end{proof}

\section{Counting and sampling realizations}

Since typically there are more than one realizations when a realization exists, and typically the number of realizations might grow exponentially, is is also a computational challenge to estimate their number and/or sample almost uniformly a solution. Here we have the following theorem.

\begin{theorem}
\begin{sloppypar}
Let $D = d_1, d_2, \dots, d_n$ and $F = f_1, f_2, \dots, f_n$ be two tree degree sequences such that each vertex is a leaf in at least one of the trees. Furthermore, assume that none of the trees is a star. Then there is an FPRAS for estimating the number of disjoint realizations and there is an FPAUS for almost uniformly sampling realizations.
\end{sloppypar}
\end{theorem}
\begin{proof}
This theorem is based on Theorem~\ref{theo:random}. As we discussed, there are random trees with at least two parallel edges. The number of pair of trees containing parallel edges $(v_{i_1},v_{j_1})$ and $(v_{i_2},v_{j_2})$ such that $d_{i_1}, d_{i_2} >1$ and $f_{j_1}, f_{j_2} > 1$ is
\begin{eqnarray}
\frac{(n-4)!}{(d_{i_1}-2)!(d_{i_2}-2)! \prod_{k \ne i_1,i_2} (d_k-1)!}&\times&\nonumber\\
\frac{(n-4)!}{(d_{j_1}-2)!(d_{j_2}-2)! \prod_{k \ne j_1,j_2} (f_k-1)!}&&.
\end{eqnarray}
Therefore, at least the same number of pair of trees have no parallel edges (that is, are edge disjoint realizations of the degree sequences) to get the expectation $1$ for the number of parallel edges. Therefore, the probability that two random trees will be edge disjoint is at least
\begin{equation*}
\frac{(d_{i_1}-1)(d_{i_2}-1)(f_{j_1}-1)(f_{j_2}-1)}{(n-2)^2(n-3)^2}.
\end{equation*}
%It means that in a polynomial (actually, $O(n^4)$) number of trials of random couple of trees, %at least one edge disjoint realization is expected.
 It follows from basic statistical considerations that an FPRAS algorithm can be designed based on this property. Indeed, let $\xi$ be the indicator variable that a random pair of trees are edge disjoint realizations. Then the number of edge disjoint realizations is
\begin{equation*}
E[\xi] \frac{(n-2)!}{\prod_{k=1}^n(d_i - 1)!}\frac{(n-2)!}{\prod_{k=1}^n(f_i - 1)!}.
\end{equation*}
Furthermore, we know that 
\begin{equation*}
E[\xi] \ge \frac{(d_{i_1}-1)(d_{i_2}-1)(f_{j_1}-1)(f_{j_2}-1)}{(n-2)^2(n-3)^2}.
\end{equation*}
Uniformly distributed random trees with a prescribed degree sequence can be generated in polynomial time based on the fact that the probability that a given leaf is connected to a vertex with degree $d_i$ is
$$
\frac{d_i-1}{n-2}.
$$
A uniformly distributed tree can be generated by randomly selecting a neighbor of a given leaf, then generating a random tree for the remaining degree sequence. Equivalently, the trees with a prescribed degree sequence can be encoded by the Pr\"uffer codes in which the index $i$ appears exactly $d_i-1$ times. Uniformly generating such Pr\"uffer codes is an elementary computational task.

Therefore, random pair of trees can be generated in polynomial time, and it is easy to check whether or not they are edge disjoint realizations. Such sampling of random trees provide an unbiased estimation for the expectation of the indicator variable $\xi$. Indeed, if $X_i$ is $1$ if the $i^\mathrm{th}$ pair of random trees are edge disjoint and $0$ otherwise, then the random variable
$$
Y_m := \sum_{i=1}^m X_i
$$
follows a binomial distribution with parameter $p = E[\xi]$ and expectation $m E[\xi]$. The tails of the binomial distributions can be bounded by the Chernoff's inequality:
\begin{equation*}
P(Y_m \le mp(1-\epsilon)) \le exp\left(-\frac{1}{2p} \frac{(mp - mp(1-\epsilon))^2}{m}\right).
\end{equation*}
This should be bounded by $\frac{\delta}{2}$ (the other half $\delta$ error will go to the other tail)
\begin{equation}
exp\left(-\frac{1}{2p} \frac{(mp - mp(1-\epsilon))^2}{m}\right) \le \frac{\delta}{2}.\label{eq:binomial-lower-tail}
\end{equation}
Solving Equation~\ref{eq:binomial-lower-tail}, we get
\begin{equation*}
m \ge \frac{-2\log\left(\frac{\delta}{2}\right)}{p\epsilon^2}.
\end{equation*}
For the upper tail, we can also use the Chernoff's inequality, just replacing $p$ with $1-p$ and the upper threshold $mp(1+\epsilon)$ with $m-mp(1+\epsilon)$:
\begin{eqnarray*}
&P(Y_m \ge mp(1+\epsilon)) \le&\nonumber \\
 &exp\left(-\frac{(m(1-p) -(m- mp(1+\epsilon)))^2}{2(1-p)m}\right).&
\end{eqnarray*}
Upper bounding this with $\frac{\delta}{2}$ and solving the inequality, we get that
\begin{equation*}
m \ge \frac{-2(1-p)\log\left(\frac{\delta}{2}\right)}{p^2\epsilon^2}.
\end{equation*}
Since $\frac{1}{p} = O(n^4)$, the necessary number of samples is indeed polynomial with the size of the problem, $\frac{1}{e}$ and $-\log(\delta)$. Furthermore, one sample can be generated in polynomial time, therefore this algorithm is indeed an FPRAS.

 It is also well known that an FPAUS algorithm can be designed in this case. The FPAUS algorithm generate $\frac{-\log(\epsilon)}{p}$ pair of random trees. If any of them is an edge disjoint realization, then the algorithm returns with it. Otherwise it generates an arbitrary realization and returns with it.

This is indeed an FPAUS algorithm, since any random pair of trees which are edge disjoint come from sharp the uniform distribution of the solutions. The probability that there will be no edge disjoint pair of trees in $m$ number of samples is
$$
(1-p)^m.
$$
This probability is not larger than $\epsilon$. Indeed,
$$
(1-p)^{\frac{-\log(\epsilon)}{p}} \le \epsilon,
$$
since
$$
\frac{-\log(\epsilon)}{p} \log(1-p) \le \log(\epsilon)
$$
because
$$
-\log(1-p) \ge p.
$$
Namely, the algorithm generates realizations from a distribution which is the convex combination $(1-\epsilon')U+\epsilon'\pi$, where $\epsilon' \le \epsilon$, $U$ is the uniform distribution and $\pi$ is an arbitrary distribution. However, the variational distance of this distribution from the uniform one is
\begin{eqnarray*}
&d_{TV}(U,(1-\epsilon')U+\epsilon'\pi) =&\nonumber\\
&\frac{1}{2}\sum_x |U(x) -((1-\epsilon')U(x)+\epsilon'\pi(x)| =&\nonumber\\
&\epsilon' \frac{1}{2}\sum_x |U(x) -\pi(x)| \le \epsilon' \le \epsilon.&
\end{eqnarray*}
Since one sample can be generated in polynomial time, and the total number of samples is polynomial with the size of the problem and $-\log(\epsilon)$, this algorithm is indeed and FPAUS.
\end{proof}

It remains an open question whether or not similar theorems exist for the case when the tree degree sequences have common high degrees. Also it is open if exact counting of the edge disjoint solutions is possible in polynomial time, although the natural conjecture is that this counting problem is $\sharpP$-complete.

\section{An NP-completeness theorem}

What can we say when only one of the two degree sequences is a tree degree sequence and the other is arbitrary? Unfortunately, we have a negative result here.

\begin{theorem}
It is $\NP$-complete to decide if there is an edge disjoint realization of a tree degree sequence and an arbitrary degree sequence. (It is not required that the tree degree sequence have a tree realization).
\end{theorem}
\begin{proof}
We use the theorem by \cite{dgm2009} that it is $\NP$-complete to decide if two bipartite degree sequences has an edge disjoint realizations. We have the following observations.
\begin{itemize}
\item A bipartite degree sequence pair
$$D =(d_{1,1}, d_{1,2}, \dots, d_{1,n_1}), (d_{2,1},d_{2,2},\dots, d_{2,n_2})$$ and
$$F= (f_{1,1}, f_{1,2}, \dots, f_{1,n_1}), (f_{2,1},f_{2,2},\dots, f_{2,n_2})$$ has an edge 
disjoint realization if and only if the simple degree sequence pair
$$D'=(d_{1,1}+n_1-1, \dots, d_{1,n_1}+n_1-1, d_{2,1},\dots, d_{2,n_2})$$ and
$$F'=(f_{1,1}, \dots, f_{1,n_1}, f_{2,1}+n_2-1,\dots, f_{2,n_2}+n_2-1)$$ has an edge disjoint realization. Indeed, if an edge disjoint bipartite realization of $D$ and $F$ is given, then the complete graph on the first vertex class can be added to the first realization and the complete graph on the second vertex class can be added to the second realization to get a (now non-bipartite) realization of $D'$ and $F'$. On the other hand, it is easy to see that any realization of $D'$ contains $K_{n_1}$ on the first $n_1$ vertices, and any realization of $F'$ contains $K_{n_2}$ on the last $n_2$ vertices. Given an edge disjoint realization of $D'$ and $F'$, deleting $K_{n_1}$ from $D'$ and $K_{n_2}$ from $F'$ yields an edge disjoint realization of $D$ and $F$.

\item The degree sequence pair $D = d_1, d_2, \dots, d_n$ and $F = f_1, f_2, \dots, f_n$ has an edge disjoint realization if and only if the degree sequence pair $D' = d_1+1, d_2+1, \dots, d_n+1, n$ and $F' = f_1, f_2, \dots, f_n, 0$ has an edge disjoint realization. Indeed, let $G_1$ and $G_2$ be an edge disjoint realization of $D$ and $F$. Then add a vertex $v_{n+1}$ to $G_1$, and connect it to all the other vertices to get a realization of $D'$. Add an isolated vertex $v_{n+1}$ to $G_2$ to get a realization of $F'$. These realizations of $D'$ and $F'$ are edge disjoint. On the other hand, in any realization of $D'$, $v_{n+1}$ is connected to all the other vertices. If edge disjoint realizations of $D'$ and $F'$ are given, delete $v_{n+1}$ from both realizations to get edge disjoint realizations of $D$ and $F$.

\item The degree sequence pair $D = d_1, d_2, \dots, d_n$ and $F = f_1, f_2, \dots, f_n$ has an edge disjoint realization if and only if the degree sequence pair $D' = d_1, d_2, \dots, d_n, 1, 1$ and $F' = f_1+1, f_2+1, \dots, f_n+1, n, 0$ has an edge disjoint realization. Indeed, any edge disjoint realization $G_1$ and $G_2$ of $D$ and $F$ can be extended to an edge disjoint realization of $D'$ and $F'$ by adding two vertices $v_{n+1}$ and $v_{n+2}$, and then connecting $v_{n+1}$ to all $v_1, \dots, v_n$ in $G_2$ and connecting $v_{n+1}$ and $v_{n+2}$ in $G_1$. On the other hand, in any edge disjoint realizations $G_1'$ and $G_2'$ of $D'$ and $F'$, $v_{n+1}$ is connected to all $v_1, \dots, v_n$ in $G_2'$, therefore, $v_{n+1}$ must be connected to $v_{n+2}$ in $G_1'$. Therefore deleting $v_{n+1}$ and $v_{n+2}$ yields an edge disjoint realization of $D$ and $F$.

\end{itemize}

We can use the first observation to prove that it is also $\NP$-complete to decide that two simple degree sequences have edge disjoint realizations. The second observation provides that it is $\NP$-complete to decide if two degree sequences  have edge disjoint realizations such that one of the degree sequences does not have 0 degrees. Finally, we can use the third observation to iteratively transform any $D$ degree sequence (that already does not have a 0 degree) to a tree degree sequence. Indeed, in each step, we add two vertices to $D$ and extend the sum of the degrees only by $2$. Therefore in a polynomial number of steps, we get a degree sequence $D'$ in which the sum of the degrees is exactly twice the number of vertices minus 2. Therefore it follows that given any bipartite degree sequences $D$ and $F$, we can construct in polynomial time two simple degree sequences $D'$ and $F'$ such that $D$ and $F$ have edge disjoint realizations if and only if $D'$ and $F'$ have edge disjoint realizations, furthermore, $D'$ is a tree degree sequence.
\end{proof}

\section{Discussion and Conclusions}\label{sec:conclusions}

In this paper, we considered packing tree degree sequences. When there are no common leaves, there are always edge disjoint caterpillar realizations. On the other hand, there might not be edge disjoint caterpillar realizations when there are common leaves, even if otherwise there are edge disjoint tree realizations.

When there are no common high degree vertices, there are edge disjoint tree realizations if and only if none of the degree sequences is a degree sequence of a star. Similar theorem exists for arbitrary number of trees, and it is easy to decide if arbitrary number of tree degree sequences without common high degrees have edge disjoint realizations.

It is also known \cite{kundu} that a degree sequence and an almost regular degree sequence have an edge disjoint realization if and only if their sum is graphical. This raises the natural question if a degree sequence and a tree sequence have edge disjoint realizations if and only if their sum is graphical. We showed that the answer is no to this question, and actually, it is $\NP$-complete to decide if an arbitrary degree sequence and a tree degree sequence have edge disjoint realizations.

We also considered to approximately count and sample edge disjoint tree realizations with prescribed degrees. We showed that it is possible if there are no common high degree vertices. It remains an open question when the two degree sequences have common high degree vertices.

\end{document}